\documentclass[10pt,a4paper]{amsart}

\pdfoutput=1
\usepackage{amsmath,amsfonts,amssymb,amsthm}
\usepackage{xcolor}
\usepackage{todonotes}
\usepackage[pagebackref=true]{hyperref}
\hypersetup{
            colorlinks   = true,
            pdfauthor    = {Jo\~{a}o Guerreiro},
            pdftitle     = {An explicit inversion formula for the $p$-adic Whittaker transform on $\text{GL}_n(\mathbb{Q}_p)$},
            pdfkeywords  = {},
            pdfcreator   = {Pdflatex},
            pdfpagemode  = UseNone,
            pdfstartview = FitH
           }
\usepackage{pdfpages}
\usepackage[normalem]{ulem}
\usepackage{url}
\usepackage{enumerate}
\usepackage{verbatim}
\usepackage{mathrsfs}
\usepackage{csquotes}
\usepackage{mathtools}
\usepackage{lipsum}
\numberwithin{equation}{section}

\theoremstyle{plain}
\newtheorem{thm}{Theorem}[section]
\newtheorem{theorem}{Theorem}[section]
\newtheorem{lemma}[thm]{Lemma}

\newtheorem{definition}[theorem]{Definition}
\newtheorem{proposition}[theorem]{Proposition}
\newtheorem{corollary}[theorem]{Corollary}
\newtheorem{remark}[theorem]{Remark}
\newcommand{\sgn}{\mathrm{sgn}}
\newcommand{\defeq}{\vcentcolon=}

\newcommand{\cc}{\mathbb{C}}
\newcommand{\qq}{\mathbb{Q}}

\theoremstyle{definition}

\title[An explicit inversion formula for the $p$-adic Whittaker transform on $\text{GL}_n(\mathbb{Q}_p)$]{An explicit inversion formula for the $p$-adic Whittaker transform on $\text{GL}_n(\mathbb{Q}_p)$}

\author{Jo\~{a}o Guerreiro}
\address{Max-Planck-Institut f\"ur Mathematik, Vivatsgasse 7, 53111, Bonn, Germany}
\email{guerreiro@mpim-bonn.mpg.de}

\date{\today}

\begin{document}

\begin{abstract}
Whittaker functions on non-archimedean fields were first introduced in the work of Jacquet, and they were characterized explicitly by Shintani. We obtain an explicit inversion formula and a Plancherel formula for the $p$-adic Whittaker transform on $\text{GL}_n(\mathbb{Q}_p)$. As an application, integral representations are obtained for the local factors of certain symmetric power $L$-functions.
\end{abstract}

\maketitle



\section{Introduction}

Whittaker functions on local non-archimedean fields (such as $\mathbb{Q}_p$) have been studied by several authors since they were introduced in 1967 by Jacquet \cite{jacquet}. Shintani \cite{shintani} obtained an explicit formula for these Whittaker functions, which was then generalized in different directions in other works (see \cite{casselmanshalika}, \cite{miyauchi}). 

Let $G= \text{GL}_n(\mathbb{Q}_p)$ for which we have the Iwasawa decomposition $G = U T K$ where $U = U(\mathbb{Q}_p)$ is the unipotent radical of the standard Borel subgroup, $T = T(\mathbb{Q}_p)$ is the torus of diagonal matrices and $K = \text{GL}_n(\mathbb{Z}_p)$ is the maximal compact subgroup.
Let $Z$ be the center of $G$. Let $\psi$ be a character on $U$ induced from a character $\psi'$ on $\mathbb{Q}_p$ via
\begin{equation*}
\psi(u) = \psi'\left(\sum_{i=1}^{n-1} u_{i,i+1}   \right), \qquad \left( \text{where }u = (u_{i,j}) \in U\right).
\end{equation*}

\begin{definition}[Spherical Hecke Algebra]
Define the spherical Hecke algebra $^K \mathcal{H}^K$ to be the set of locally constant, compactly supported functions $f : \text{GL}_n(\mathbb{Q}_p) \rightarrow \cc$ satisfying
\begin{equation*}
f( k_1 g k_2) = f(g)
\end{equation*}
for $k_1, k_2 \in K$, $g \in \text{GL}_n(\mathbb{Q}_p)$.
\end{definition}

On $\text{GL}_n(\mathbb{Q}_p)$, Whittaker functions transform my the action of a character under the unipotent radical and satisfy an integral property relative to elements of the spherical Hecke algebra, as below.

\begin{definition}[Whittaker Function on $\text{GL}_n(\mathbb{Q}_p)$]
A Whittaker function on $\text{GL}_n(\mathbb{Q}_p)$ is a $K$-finite smooth function of moderate growth $W : \text{GL}_n(\mathbb{Q}_p) \rightarrow \mathbb{C}$ that satisfies
\begin{equation*}
W (u g) = \psi(u) W(g)
\end{equation*}
for some fixed character $\psi$ as defined above and any $u \in U$, and
\begin{equation*}
\int \limits_{G} W(g x) \phi(x) \mathop{d x} = \lambda(\phi) W(g)
\end{equation*}
for each algebra homomorphism $\lambda : \, ^K\!\mathcal{H}^K \rightarrow \cc$, $\phi \in \, ^K\!\mathcal{H}^K$ and $g \in \text{GL}_n(\mathbb{Q}_p)$.
\end{definition}

Let $h:Z(\mathbb{Q}_p)\backslash T(\mathbb{Q}_p) \to \mathbb{C}$ be a smooth function. The $p$-adic Whittaker transform of $h$ is defined to be
\begin{equation*}
\int\limits_{Z(\mathbb{Q}_p)\backslash T(\mathbb{Q}_p)} h(t) W(t) d^\times t,
\end{equation*}
where $W$ is a Whittaker function on $\text{GL}_n(\mathbb{Q}_p)$.

The main goal of the present paper is to obtain an inversion formula for the $p$-adic Whittaker transform. The precise inversion formula for the Whittaker transform is given in Theorem \ref{transformthm}.
On archimedean local fields Whittaker functions have been thoroughly studied and an inversion formula for the Whittaker transform in known due to the works of Wallach \cite{wallach} and Goldfeld-Kontorovich \cite{transform}. The approach presented in this paper is inspired by the latter. A crucial step in our approach is the computation of an integral of a product of two $p$-adic Whittaker functions.
One can view this computation as a non-archimedean analogue of Stade's formula \cite{stade}. As a corollary of the inversion formula we also obtain a Plancherel formula for the $p$-adic Whittaker transform. 

In the last section, we use the inversion formula to write integral representations of local $L$-factors associated to a symmetric square and symmetric cube lifts.

We remark that the inversion formula had already been obtained in the setting of $p$-adic reductive groups by Delorme \cite{delorme}. However, the approach presented here provides a more explicit presentation of the result in this particular setting and its simple derivation solely relies on complex analysis.

\section{Inversion Formula}

In this section we provide all the necessary definitions and results that allow us to state the main theorem \ref{transformthm}, which gives the inverse formula for the $p$-adic Whittaker transform.

We start by noting Whittaker functions on $\text{GL}_n(\mathbb{Q}_p)$ arise naturally from certain irreducible representation of this group.

\begin{definition}[Whittaker Function associated to a Representation $\pi$]
\label{def}
Let $V$ be a complex vector space and let $(\pi,V)$ be an irreducible generic representation of $G$. Then the space 
\begin{equation*}
\text{Hom}_{U}(\pi,\psi) = \left\{ f : V \rightarrow \mathbb{C} \text{ linear}: f(\pi(u) v) = \psi(u) f(v), \forall v \in V, u \in U \right\}
\end{equation*}
is one-dimensionanal generated by an element $\lambda$. For $v \in V$ a newform, define a Whittaker function $W$ associated to $\pi$ as
\begin{equation*}
W(g) = \lambda(\pi(g)v).
\end{equation*}
\end{definition}

\begin{remark}
As the space of newforms in $V$ is one-dimensional then the Whittaker function associated to a representation $\pi$ is well-defined up to a constant. Furthermore, a Whittaker function associated to an irreducible generic representation $\pi$ of $\text{GL}_n(\mathbb{Q}_p)$ is a nonzero Whittaker function on $\text{GL}_n(\mathbb{Q}_p)$. For proofs of these statements see \cite{jacquetetal} and \cite{kazhdan}.
\end{remark}

In order to explicitly describe Whittaker functions associated to certain representations of $\text{GL}_n(\mathbb{Q}_p)$ we proceed to quote a theorem of Shintani \cite{shintani}. For $g = z u t k \in \text{GL}_n(\mathbb{Q}_p)$ define the Iwasawa coordinates

\begin{equation*}
 g = z u \left( \begin{array}{cccc}
t_1 \cdots t_{n-1} &  &  & \\
 & \ddots & &  \\
 & & t_1 & \\
 &  & & 1 \end{array} \right) k
\end{equation*}
where $z \in Z$, $u \in U$, $t \in T$ and $k \in K$.

\begin{theorem}[Shintani]
\label{shintani}
Let $\pi$ be an irreducible unramified generic representation of $\text{GL}_n(\mathbb{Q}_p)$ and $W_{\pi}$ the Whittaker function associated to a representation $\pi$ (normalized such that $W_{\pi}(\mathrm{id}) = 1$ where $\mathrm{id}$ denotes the $n \times n$ identity matrix) as in Definition \ref{def}. We can write the L-function associated to $\pi$ as
\begin{equation*}
L(s,\pi) = \prod_{i=1}^n \left(1-\alpha_i p^{-s} \right)^{-1}
\end{equation*}
where $\alpha_i \in \mathbb{C}$ are all nonzero and $\alpha_1 \cdots \alpha_n = 1$. 
Then, for
\begin{equation*}
t = \left( \begin{array}{cccc}
t_1 \cdots t_{n-1} &  &  & \\
 & \ddots & &  \\
 & & t_1 & \\
 &  & & 1 \end{array} \right)
 \end{equation*}
we have
\begin{equation*}
W_{\pi}(t) = \begin{cases} \delta^{1/2}(t) s_{\text{\rm-}\!\log|t|_p}(\alpha), & \mbox{if } t_i \in \mathbb{Z}_p \, \text{ for } i=1,\cdots,n-1, \\ 0, & \mbox{otherwise;}  \end{cases}
\end{equation*}
where 
\begin{equation*}
\delta(t) = \prod_{k=1}^{n-1} |t_k|_p^{k(n-k)},
\end{equation*}

\begin{equation*}
\log|t|_p = \left(\log|t_1|_p, \cdots, \log|t_{n-1}|_p \right) \quad \text{for} \quad \log = \log_p,
\end{equation*}
and
\begin{equation*}
s_{\lambda}(\alpha) = \frac{\left| \begin{array}{cccc}
\alpha_1^{n-1+\lambda_1+\lambda_2+\cdots+\lambda_{n-1}}  & \alpha_2^{n-1+\lambda_1+\lambda_2+\cdots+\lambda_{n-1}} & \cdots & \alpha_n^{n-1+\lambda_1+\lambda_2+\cdots+\lambda_{n-1}} \\
\alpha_1^{n-2+\lambda_1+\lambda_2+\cdots+\lambda_{n-2}}  & \alpha_2^{n-2+\lambda_1+\lambda_2+\cdots+\lambda_{n-2}} & \cdots &  \alpha_n^{n-2+\lambda_1+\lambda_2+\cdots+\lambda_{n-2}} \\
\vdots & \vdots & \ddots &  \vdots \\
\alpha_1^{1+\lambda_1} & \alpha_2^{1 +\lambda_1} & \cdots & \alpha_n^{1 +\lambda_1} \\
1 & 1 & \cdots & 1 \end{array} \right|}{\left| \begin{array}{cccc}
\alpha_1^{n-1}  & \alpha_2^{n-1} & \cdots & \alpha_n^{n-1} \\
\alpha_1^{n-2}  & \alpha_2^{n-2} & \cdots & \alpha_n^{n-2} \\
\vdots & \vdots & \ddots & \vdots \\
\alpha_1 & \alpha_2 & \cdots & \alpha_n \\
1 & 1 & \cdots & 1 \end{array} \right|}
\end{equation*}
is the Schur polynomial, where $\lambda = (\lambda_1,\cdots,\lambda_n)$ and $\lambda_i \in \mathbb{Z}$.
\end{theorem}

\begin{proof}
See \cite{shintani}. For the relation between the $L$-function $L(s,\pi)$ and the Whittaker function $W_{\pi}$ see Subsection 3.1.3~\cite{cogdell}.
\end{proof}

\begin{remark}Throughout this paper we will always use $\log$ to denote the base $p$ logarithm.
We will also use the notation $W_{\alpha}$, for $\alpha \in (\mathbb{C} \backslash \{0 \})^n$, to denote the function
\begin{equation*}
W_{\alpha}(t) = \begin{cases} \delta^{1/2}(t) s_{\text{\rm-}\!\log|t|_p}(\alpha), & \mbox{if } t_i \in \mathbb{Z}_p \, \text{ for } i=1,\cdots,n-1, \\ 0, & \mbox{otherwise.}  \end{cases}.
\end{equation*}finite
This coincides with the definition of $W_{\pi}$ when $\alpha = (\alpha_1,\cdots, \alpha_n)$ are the Langlands parameters associated to $\pi$.
\end{remark}
We are now able to define the $p$-adic Whittaker transform precisely.

\begin{definition}[$p$-adic Whittaker transform]
\label{whittakertransform}
Let $h:Z(\mathbb{Q}_p) \backslash T(\mathbb{Q}_p) \rightarrow \mathbb{C}$. Let $\mathcal{S}_{\eta}$ be the annulus $\{ z \in \cc : p^{-\eta} < |z|_{\cc} < p^{\eta} \}$, for some $\eta >0$. Define the Whittaker transform $h^{\sharp} : \mathcal{S}_{\eta}^{n} \rightarrow \mathbb{C}$ by
\begin{equation*}
h^{\sharp}(\alpha) = \int \limits_{T(\mathbb{Z}_p) \backslash T(\mathbb{Q}_p)} h(t) W_{\alpha}(t) d^{\times}t,
\end{equation*}
provided the integral converges absolutely,
where
\begin{equation*}
d^{\times}t = \prod_{k=1}^{n-1} |t_k|_p^{-k(n-k)} \frac{dt_k}{|t_k|_p}.
\end{equation*}
\end{definition}

\begin{remark}
\label{convergence}
A sufficient condition for the absolute convergence of the above integral is that $h$ satisfies the decay condition
\begin{equation*}
\lvert h(t) \rvert_{\cc} \ll \delta^{1/2}(t) \lvert t_1^{n-1} t_2^{n-2} \cdots t_{n-1}\rvert_p^{\eta + \varepsilon}
\end{equation*}
for any $\varepsilon > 0$.
\end{remark}

\begin{definition}[Inverse Whittaker Transform]
\label{invtransform}
Let $H : \mathcal{S}_{\eta}^n \rightarrow \mathbb{C}$ be a holomorphic symmetric function, for some $\eta >0$. Let
\begin{equation*}
W_{\alpha}(t) = \begin{cases} \delta^{1/2}(t) s_{\text{\rm-}\!\log|t|_p}(\alpha), & \mbox{if } t_i \in \mathbb{Z}_p \, \text{ for } i=1,\cdots,n-1, \\ 0, & \mbox{otherwise;}  \end{cases}
\end{equation*}
where $\alpha = (\alpha_1,\ldots,\alpha_n)$. Define the inverse Whittaker transform of $H$, $H^{\flat}: Z(\mathbb{Q}_p) \backslash T(\mathbb{Q}_p) \rightarrow \mathbb{C}$ by
\begin{equation*}
H^{\flat}(t) = \frac{1}{n! (2\pi i)^{n-1}} \int \limits_{\mathbb{T}} H(\beta) W_{1/\beta}(t) \prod \limits_{\substack{i, j = 1 \\ i \neq j}}^n \left(\beta_i - \beta_j \right) \frac{d \beta_1 \cdots d \beta_{n-1}}{\beta_1 \cdots \beta_{n-1}}
\end{equation*}
where  $\mathbb{T} = \big\{ (\beta_1,\cdots,\beta_{n-1}) \in \mathbb{C}^{n-1} : |\beta_i|_{\mathbb{C}} = 1 \big\}$,
$1/\beta = \left(1/\beta_1,\cdots, 1/\beta_n\right)$ and $\beta_n = \frac{1}{\beta_1 \cdots \beta_{n-1}}$.
\end{definition}

\begin{remark}
Note that $H^{\flat}$ is supported on $(\mathbb{Z}_p \backslash \{0\})^{n-1}$, if we identify $Z(\mathbb{Q}_p) \backslash T(\mathbb{Q}_p)$ with $(\mathbb{Q}_p^{\times})^{n-1}$ as in Theorem~\ref{shintani}. Moreover, $H^{\flat}$ is invariant under the action of $(\mathbb{Z}_p^{\times})^{n-1}$.
\end{remark}

The inverse transform is given in the following theorem.

\begin{theorem}[Inversion Formula]
\label{transformthm}
Let $H : \mathcal{S}_{\eta}^n \rightarrow \mathbb{C}$ be a holomorphic symmetric function, for some $\eta >0$. Assume $(H^{\flat})^{\sharp}$ converges absolutely on $\mathcal{S}_{\eta}^n$. Then,
\begin{equation*}
(H^{\flat})^{\sharp}(\alpha) = H(\alpha)
\end{equation*}
for $\alpha \in \mathcal{S}_{\eta}^n$ with $\alpha_1 \cdots \alpha_n = 1$.
\end{theorem}

By studying the image of the map $H \mapsto H^{\flat}$ we obtain the following corollary.


\begin{corollary}
\label{corollary}
Let $h : Z(\mathbb{Q}_p) \backslash T(\mathbb{Q}_p) \rightarrow \mathbb{C}$. For
\begin{equation*}
t = \left( \begin{array}{cccc}
t_1 \cdots t_{n-1} &  &  & \\
 & \ddots & &  \\
 & & t_1 & \\
 &  & & 1 \end{array} \right)
\end{equation*}
 assume that $h(t)$ is supported on the region $\{ t_i \in \mathbb{Z}_p \backslash \{ 0 \} : i=1,\cdots,n-1 \}$, that the integral which defines $h^{\sharp}$ (in Definition~\ref{whittakertransform}) is absolutely convergent, and that $h(t) = f(-|t_1|_p,\cdots,-|t_{n-1}|_p)$ for some function $f : \mathbb{Z}_{\geq 0}^{n-1} \rightarrow \mathbb{C}$. Then
\begin{equation*}
(h^{\sharp})^{\flat} = h.
\end{equation*}
\end{corollary}

We postpone the proofs of Theorem~\ref{transformthm} and Corollary~\ref{corollary} to the next section. As a consequence of Corollary~\ref{corollary} we have the Plancherel formula for the $p$-adic Whittaker transform.

\begin{corollary}[Plancherel formula]
Let $h_1, h_2$ be two functions that satisfy the conditions of Corollary~\ref{corollary}. We have,
\begin{equation*}
\langle h_1, h_2 \rangle = \langle h_1^{\sharp}, h_2^{\sharp} \rangle
\end{equation*}
where
\begin{align*}
\langle h_1, h_2 \rangle & \defeq \int \limits_{T(\mathbb{Z}_p) \backslash T(\mathbb{Q}_p)} h_1(t) \overline{h_2(t)}  d^{\times}t, \\
\langle h_1^{\sharp}, h_2^{\sharp} \rangle & \defeq \frac{1}{n! (2\pi i)^{n-1}} \int \limits_{\mathbb{T}} h_1^{\sharp}(\beta) \overline{h_2^{\sharp}(\beta)} \prod \limits_{\substack{i, j = 1 \\ i \neq j}}^n \left(\beta_i - \beta_j \right) \frac{d \beta_1 \cdots d \beta_{n-1}}{\beta_1 \cdots \beta_{n-1}}.
\end{align*}
\end{corollary}

\section{Proof of Theorem~\ref{transformthm}}
We start the section with the proof of two of the ingredients needed to prove the inverse transform theorem. Both results can also be found in Section 4 of \cite{bump}. They are included here as their proofs are concise and improve the exposition. The first of these is a version of Cauchy's identity.

\begin{lemma}[Cauchy's Identity]
\label{cauchy}
Let $s_{\lambda}$ be the Schur polynomial defined in the previous section. Let $\alpha_1,\cdots,\alpha_n,\beta_1,\cdots,\beta_n \in \mathbb{C}$ such that $|\alpha_i \beta_j| < 1$ for every $1 \leq i, j \leq n$. Then,
\begin{equation*}
\sum \limits_{m_1,\cdots, m_{n-1} \geq 0} s_{\mathbf{m}}\left(\alpha\right) s_{\mathbf{m}}\left(\beta\right) \\
 = \frac{1-\alpha_1 \cdots \alpha_n \beta_1 \cdots \beta_n}{\prod \limits_{i,j=1}^n \left(1-\alpha_i \beta_j \right)}
 \end{equation*}
 where $\mathbf{m} = (m_1,\cdots,m_{n-1})$
\end{lemma}

\begin{proof}
This proof follows \cite{goldfeldbook} and \cite{macdonald}.
We start by stating the Cauchy determinant identity (lemma 7.4.18 in \cite{goldfeldbook})
\begin{equation*} \left| \begin{array}{ccc}
\frac{1}{1-\alpha_1 \beta_1}  & \cdots & \frac{1}{1-\alpha_n \beta_1} \\
\vdots  & \ddots &  \vdots \\
\frac{1}{1-\alpha_1 \beta_n}  & \cdots & \frac{1}{1-\alpha_n \beta_n} \\ \end{array} \right| = \frac{\prod \limits_{1 \leq i < j \leq n} (\alpha_i - \alpha_j) \prod \limits_{1 \leq i < j \leq n} (\beta_i - \beta_j)}{\prod \limits_{i,j = 1}^n (1 - \alpha_i \beta_j)}.
\end{equation*}
Expanding each term $\frac{1}{1-\alpha_i \beta_j}$ as $1 + \alpha_i \beta_j + \alpha_i^2 \beta_j^2 + \cdots$, we obtain

\begin{equation*}
\left|  \begin{array}{ccc}
\frac{1}{1-\alpha_1 \beta_1}  & \cdots & \frac{1}{1-\alpha_n \beta_1} \\
\vdots  & \ddots &  \vdots \\
\frac{1}{1-\alpha_1 \beta_n} & \cdots & \frac{1}{1-\alpha_n \beta_n} \\ \end{array} \right| = \sum \limits_{l_1,\cdots,l_n \geq 0} \, \sum \limits_{\sigma \in S_n} \sgn(\sigma) \alpha_{\sigma(1)}^{l_1} \cdots \alpha_{\sigma(n)}^{l_n} \beta_1^{l_1} \cdots \beta_n^{l_n}.
\end{equation*}

Notice that if the $l_i$ are not all distinct then the sum over $S_n$ vanishes. Therefore
\begin{align*}
\left| \begin{array}{ccc}
\frac{1}{1-\alpha_1 \beta_1}  & \cdots & \frac{1}{1-\alpha_n \beta_1} \\
\vdots  & \ddots &  \vdots \\
\frac{1}{1-\alpha_1 \beta_n}  & \cdots & \frac{1}{1-\alpha_n \beta_n} \\ \end{array} \right| & = \sum \limits_{\substack{l_1 > \cdots > l_n \geq 0 \\ \sigma, \tau \in S_n}} \sgn(\sigma) \sgn(\tau) \alpha_{\sigma(1)}^{l_1} \cdots \alpha_{\sigma(n)}^{l_n} \beta_{\tau(1)}^{l_1} \cdots \beta_{\tau(n)}^{l_n} \\
& = \sum \limits_{l_1 > \cdots > l_n \geq 0} \left| \begin{array}{ccc}
\alpha_1^{l_1}  & \cdots & \alpha_n^{l_1} \\
\vdots  & \ddots &  \vdots \\
\alpha_1^{l_n}  & \cdots & \alpha_n^{l_n} \\ \end{array} \right| \left| \begin{array}{ccc}
\beta_1^{l_1}  & \cdots & \beta_n^{l_1} \\
\vdots  & \ddots &  \vdots \\
\beta_1^{l_n}  & \cdots & \beta_n^{l_n} \\ \end{array} \right|.
\end{align*}
By making the substitution $(l_1,\cdots,l_n) = (m_0+\cdots+m_{n-1}+(n-1), m_0+\cdots+m_{n-2}+(n-2), \cdots, m_0)$ and dividing by $\prod \limits_{1 \leq i < j \leq n} (\alpha_i - \alpha_j) \prod \limits_{1 \leq i < j \leq n} (\beta_i - \beta_j)$ we obtain
\begin{equation*}
 \frac{1}{\prod \limits_{i,j=1}^n \left(1-\alpha_i \beta_j \right)} = \sum \limits_{m_0, m_1,\cdots, m_{n-1} \geq 0} s_{\bf{m}} \left(\alpha\right) s_{\bf{m}}\left(\beta\right) (\alpha_1 \cdots \alpha_n \beta_1 \cdots \beta_n)^{m_0},
 \end{equation*}
where right hand side simplifies to
\begin{equation*}
\frac{1}{1-\alpha_1 \cdots \alpha_n \beta_1 \cdots \beta_n} \sum \limits_{m_1,\cdots, m_{n-1} \geq 0} s_{\bf{m}}\left(\alpha\right) s_{\bf{m}}\left(\beta\right).
\end{equation*}
Multiplying out by $1-\alpha_1 \cdots \alpha_n \beta_1 \cdots \beta_n$ concludes the proof.
\end{proof}

The second necessary ingredient is an identity for the integral of a product of two Whittaker function, as in \cite{stade}.
\begin{proposition}
\label{innerprod}
Let $\alpha_1,\cdots,\alpha_n,\beta_1,\cdots,\beta_n \in \mathbb{C}$ such that $|\alpha_i \beta_j| \leq 1$ for every $1 \leq i, j \leq n$, and $\varepsilon > 0$. Let $W_{\alpha}$, $W_{\beta}$ be Whittaker functions as defined in Definition \ref{invtransform}. Then
\begin{equation*}
\int \limits_{Z(\mathbb{Q}_p) \backslash T(\mathbb{Q}_p)} W_{\alpha}(t) W_{\beta}(t) \prod_{k=1}^{n-1} |t_k|_p^{\varepsilon(n-k)}  d^{\times}t = \frac{1-\frac{\alpha_1 \cdots \alpha_n \beta_1 \cdots \beta_n}{p^{\varepsilon n}}}{\prod \limits_{i,j=1}^n \left(1-\frac{\alpha_i \beta_j}{p^{\varepsilon}} \right)} 
\end{equation*}
\end{proposition}

\begin{proof}
We start by applying Theorem \ref{shintani} to write down the Whittaker functions explicitly:
\begin{align*}
& \int \limits_{Z(\mathbb{Q}_p) \backslash T(\mathbb{Q}_p)} W_{\alpha}(t) W_{\beta}(t) \prod_{k=1}^{n-1} |t_k|_p^{\varepsilon(n-k)}  d^{\times}t \\
& = \int \limits_{\mathbb{Z}_p^{n-1}} \delta(t) s_{\text{\rm-}\!\log|t|_p}(\alpha) s_{\text{\rm-}\!\log|t|_p}\left(\beta\right) \prod_{k=1}^{n-1} |t_k|_p^{-(k-\varepsilon)(n-k)} \frac{dt_k}{|t_k|_p} \\ 
& = \int \limits_{\mathbb{Z}_p^{n-1}} s_{\text{\rm-}\!\log|t|_p}(\alpha) s_{\text{\rm-}\!\log|t|_p}\left(\beta\right) \prod_{k=1}^{n-1} |t_k|_p^{\varepsilon(n-k)} \frac{dt_k}{|t_k|_p}.
\end{align*}
Breaking up the region of integration by absolute value we get
\begin{align*}
& \int \limits_{Z(\mathbb{Q}_p) \backslash T(\mathbb{Q}_p)} W_{\alpha}(t) W_{\beta}(t) \prod_{k=1}^{n-1} |t_k|_p^{\varepsilon(n-k)}  d^{\times}t \\
& = \sum \limits_{m_1,\cdots, m_{n-1} \geq 0} \, \, \overbrace{\idotsint}^{n-1 \text{ times}} \limits_{p^{m_k} \mathbb{Z}_p^{\times}}  s_{\text{\rm-}\!\log|t|_p}(\alpha) s_{\text{\rm-}\!\log|t|_p}\left(\beta\right) \prod_{k=1}^{n-1} |t_k|_p^{\varepsilon(n-k)} \frac{dt_k}{|t_k|_p} \\
& = \sum \limits_{m_1,\cdots, m_{n-1} \geq 0} s_{\mathbf{m}}(\alpha) s_{\mathbf{m}}\left(\beta \right) p^{-\varepsilon(n-1)m_1 -\varepsilon(n-2)m_2 + \cdots -\varepsilon m_{n-1}} \\
& = \sum \limits_{m_1,\cdots, m_{n-1} \geq 0} s_{\mathbf{m}}\left(\frac{\alpha}{p^{\varepsilon}}\right) s_{\mathbf{m}}\left(\beta\right) \\
& = \frac{1-\frac{\alpha_1 \cdots \alpha_n \beta_1 \cdots \beta_n}{p^{\varepsilon n}}}{\prod \limits_{i,j=1}^n \left(1-\frac{\alpha_i \beta_j}{p^{\varepsilon}} \right)} 
\end{align*}
where the last equality follows from Cauchy's identity \ref{cauchy}.
\end{proof}

We are now ready to proceed to the proof of the Theorem~\ref{transformthm}.

\begin{proof}[Proof of Theorem~\ref{transformthm}]
Restate the theorem as \begin{equation*}
H(\alpha) = \int \limits_{Z(\mathbb{Q}_p) \backslash T(\mathbb{Q}_p)} H^{\flat}(t) W_{\alpha}(t) d^{\times}t.
\end{equation*}
Recall that $\alpha \in \mathcal{S}_{\eta}^n$ and $\alpha_1 \cdots \alpha_n = 1$. Further assume that $|\alpha_1|_{\mathbb{C}} = \cdots = |\alpha_{n}|_{\mathbb{C}} = 1$. We can later remove this assumption as both sides of the equality are holomorphic functions. Note that $H$ is invariant under permutations of $(\alpha_1,\cdots,\alpha_n)$ as the Whittaker function also has those symmetries.
Define
\begin{equation*}
H_{\varepsilon}(\alpha) = \int \limits_{Z(\mathbb{Q}_p) \backslash T(\mathbb{Q}_p)}H^{\flat}(t) W_{\alpha}(t) \prod_{k=1}^{n-1} |t_k|_p^{\varepsilon(n-k)} d^{\times}t.
\end{equation*}
Then
\begin{equation*}
\lim \limits_{\varepsilon \rightarrow 0} H_{\varepsilon}(\alpha) = \int \limits_{Z(\mathbb{Q}_p) \backslash T(\mathbb{Q}_p)} H^{\flat}(t) W_{\alpha}(t) d^{\times}t.
\end{equation*}
It suffices to show that $ \lim \limits_{\varepsilon \rightarrow 0} H_{\varepsilon}(\alpha) = H(\alpha)$ as well. \\
Replacing $H^{\flat}(t)$ by its explicit formula and swapping the order of integration we get
\begin{multline*}
H_{\varepsilon}(\alpha) = \frac{1}{n!(2\pi i)^{n-1}} \int \limits_{\mathbb{T}} \left( \, \int \limits_{Z(\mathbb{Q}_p) \backslash T(\mathbb{Q}_p)} W_{\alpha}(t) W_{1/\beta}(t) \prod_{k=1}^{n-1} |t_k|_p^{\varepsilon(n-k)}  d^{\times}t \right) \times \\
H(\beta) \prod \limits_{\substack{i, j = 1 \\ i \neq j}}^n \left(\beta_i - \beta_j \right) \frac{d \beta_1 \cdots d \beta_{n-1}}{\beta_1 \cdots \beta_{n-1}} 
\end{multline*}
We use Proposition \ref{innerprod} to compute the innermost integrals and get
\begin{equation*}
H_{\varepsilon}(\alpha) = \frac{1}{n!(2\pi i)^{n-1}} \int \limits_{\mathbb{T}} \frac{ H(\beta) \left(1-\frac{1}{p^{\varepsilon n}}\right)}{\prod \limits_{i,j=1}^n \left(\beta_j-\frac{\alpha_i}{p^{\varepsilon}} \right)} \prod \limits_{\substack{i, j = 1 \\ i \neq j}}^n \left(\beta_i - \beta_j \right) \frac{d \beta_1 \cdots d \beta_{n-1}}{\beta_1 \cdots \beta_{n-1}}
\end{equation*}
Now shift each variable $\beta_i$ ($i=1,\cdots,n-1$) to the contour given by the equation $|z|_{\mathbb{C}} = \frac{1}{p^{\varepsilon+\varepsilon'}}$, in order. For the sake of clearness, we shall assume that all the $\alpha_i$ are distinct.
When the contour of integration for $\beta_1$ is shifted one picks up several residues at
\begin{equation*}
\beta_1 = \frac{\alpha_1}{p^{\varepsilon}}, \cdots, \beta_1 = \frac{\alpha_{n}}{p^{\varepsilon}}.
\end{equation*}
For which residue integral obtained in this manner one can shift the contour of integration for $\beta_2$ picking up $n-1$ residues in the process. Repeating this process for which $\beta_i$, $i=3,\cdots,n-1$ one obtains
\begin{equation*}
H_{\varepsilon}(\alpha) = \frac{1}{n!} \sum \limits_{\sigma \in S_n} \mathcal{R}_{\varepsilon}^{\sigma} + \mathcal{I}_{\varepsilon,\varepsilon'}
\end{equation*}this
where
\begin{equation*}
\mathcal{R}_{\varepsilon}^{\sigma} = \frac{H\left(\frac{\alpha_{\sigma(1)}}{p^{\varepsilon}},\cdots,\frac{\alpha_{\sigma(n)}}{p^{\varepsilon}} \right) \left(1-\frac{1}{p^{\varepsilon n}}\right) \prod \limits_{i=1}^{n-1} \left( \frac{\alpha_{\sigma(i)}}{p^{\varepsilon}} - p^{(n-1)\varepsilon} \alpha_{\sigma(n)} \right)}{\left(p^{(n-1)\varepsilon} \alpha_{\sigma(n)} - \frac{\alpha_{\sigma(n)}}{p^{\varepsilon}}\right) \prod \limits_{i=1}^{n-1} \left( \frac{\alpha_{\sigma(i)}}{p^{\varepsilon}} - \frac{\alpha_{\sigma(n)}}{p^{\varepsilon}} \right) \prod \limits_{i=1}^{n-1} \frac{\alpha_{\sigma(i)}}{p^{\varepsilon}}}
\end{equation*}
and
$\mathcal{I}_{\varepsilon,\varepsilon'}$ is a sum of residue integrals with each integrand bounded (in absolute value) by

\begin{equation*}
C_H \frac{ \left(1 - \frac{1}{p^{n\varepsilon}} \right) p^{\varepsilon(n^3+n)}}{1 - \frac{1}{p^{\varepsilon'}}}
\end{equation*}
for some constant $C_H > 0$ depending only on the function $H$.
Therefore,
\begin{align*}
\lim \limits_{\varepsilon \rightarrow 0} \mathcal{R}_{\varepsilon}^{\sigma} 
= & \lim \limits_{\varepsilon \rightarrow 0} \frac{H\left(\frac{\alpha_{\sigma(1)}}{p^{\varepsilon}},\cdots,\frac{\alpha_{\sigma(n)}}{p^{\varepsilon}} \right)}{\prod \limits_{i=1}^{n} \frac{\alpha_{\sigma(i)}}{p^{\varepsilon}}} \frac{ \left(1-\frac{1}{p^{\varepsilon n}}\right)}{\left(p^{(n-1)\varepsilon} - \frac{1}{p^{\varepsilon}}\right)} \frac{\prod \limits_{i=1}^{n-1} \left( \frac{\alpha_{\sigma(i)}}{p^{\varepsilon}} - p^{(n-1)\varepsilon} \alpha_{\sigma(n)} \right)}{\prod \limits_{i=1}^{n-1} \left( \frac{\alpha_{\sigma(i)}}{p^{\varepsilon}} - \frac{\alpha_{\sigma(n)}}{p^{\varepsilon}} \right)} \\
& = \frac{H\left(\alpha_{\sigma(1)},\cdots,\alpha_{\sigma(n)} \right)}{\prod \limits_{i=1}^{n} \alpha_{\sigma(i)}} \frac{\prod \limits_{i=1}^{n-1} \left( \alpha_{\sigma(i)} - \alpha_{\sigma(n)} \right)}{\prod \limits_{i=1}^{n-1} \left( \alpha_{\sigma(i)} - \alpha_{\sigma(n)} \right)} \\
& = H\left(\alpha_{\sigma(1)},\cdots,\alpha_{\sigma(n)} \right) \\
& = H(\alpha)
\end{align*}
and
\begin{equation*}
\lim \limits_{\varepsilon \rightarrow 0} |\mathcal{I}_{\varepsilon,\varepsilon'}| \ll_{H} \lim \limits_{\varepsilon \rightarrow 0} \frac{ \left(1 - \frac{1}{p^{n\varepsilon}} \right) p^{\varepsilon(n^3+n)}}{1 - \frac{1}{p^{\varepsilon'}}} = 0.
\end{equation*}
In conclusion, we obtain
\begin{equation*}
\lim \limits_{\varepsilon \rightarrow 0} H_{\varepsilon}(\alpha) = \frac{1}{n!} \sum \limits_{\sigma \in S_n} H(\alpha) = H(\alpha)
\end{equation*}
as desired.
If the $\alpha_i$ are not all distinct then the original integrand would have higher order poles. This means that the residue sum would have a small number of residue but some residues would contribute to the sum with a multiple of $H(\alpha)$.
\end{proof}

To prove the corollary one simply needs to compute the image of the inverse transform $H \mapsto H^{\flat}$ as any function $h$ in the image of this mapping will satisfy $(h^{\sharp})^{\flat} = h$. Simply choose $H$ such that $h = H^{\flat}$, and apply Theorem~\ref{transformthm} to $H$, to obtain $(h^{\sharp})^{\flat} = ((H^{\flat})^{\sharp})^{\flat} = H^{\flat} = h$.

\begin{proof}[Proof of Corollary~\ref{corollary}]
We start by noting that any function $h$ satisfying the assumptions in the corollary is a sum of scalar multiples of functions of the form:

\begin{equation*}
f_{\lambda_1,\cdots,\lambda_{n-1}}(t) = \begin{cases} 1, & \mbox{if } \log |t_i|_p = -\lambda_i \text{ for } i=1,\cdots,n-1, \\ 0, & \mbox{otherwise.}  \end{cases}
\end{equation*}
where $\lambda_1,\cdots,\lambda_{n-1}$ are nonnegative integers. Therefore, it suffices to show that all such functions are in the image of the map of the inverse transform $H \mapsto H^{\flat}$.
Let $H(\beta) = \frac{W_{\beta}(p^\lambda)}{\delta(p^{\lambda})}$ where
\begin{equation*}
p^{\lambda} = \left( \begin{array}{cccc}
p^{\lambda_1} \cdots p^{\lambda_{n-1}} &  &  & \\
 & \ddots & &  \\
 & & p^{\lambda_1} & \\
 &  & & 1 \end{array} \right).
\end{equation*}
Then
\begin{align*}
H^{\flat}(t) & = \frac{1}{\delta(p^{\lambda}) n! (2\pi i)^{n-1}} \int \limits_{\mathbb{T}} W_{\beta}(p^\lambda) W_{1/\beta}(t) \prod \limits_{\substack{i, j = 1 \\ i \neq j}}^n \left(\beta_i - \beta_j \right) \frac{d \beta_1 \cdots d \beta_{n-1}}{\beta_1 \cdots \beta_{n-1}} \\
& = \frac{\delta^{1/2}(t)}{\delta^{1/2}(p^{\lambda}) n! (2\pi i)^{n-1}} \int \limits_{\mathbb{T}} s_{\lambda}(\beta) s_{\text{\rm-}\!\log|t|_p}(1/\beta) \prod \limits_{\substack{i, j = 1 \\ i \neq j}}^n \left(\beta_i - \beta_j \right) \frac{d \beta_1 \cdots d \beta_{n-1}}{\beta_1 \cdots \beta_{n-1}}.
\end{align*}
By the definition of Schur polynomials we further infer that
\begin{align*}
H^{\flat}(t) & = \frac{\delta^{1/2}(t)}{\delta^{1/2}(p^{\lambda}) n! (2\pi i)^{n-1}} \int \limits_{\mathbb{T}} \sum \limits_{\sigma, \tau \in S_n} \left(\beta_{\sigma(1)}^{(n-1)+\lambda_1+\cdots+\lambda_{n-1}} \cdots \beta_{\sigma(n-1)}^{1+\lambda_1} \right) \times \\
&  \left( \beta_{\tau(1)}^{-(n-1)+\log|t_1|_p+\cdots+\log|t_{n-1}|_p} \cdots \beta_{\tau(n-1)}^{-1+\log|t_1|_p} \right) \frac{d \beta_1 \cdots d \beta_{n-1}}{\beta_1 \cdots \beta_{n-1}} \\
& = \frac{\delta^{1/2}(t)}{\delta^{1/2}(p^{\lambda}) n! (2\pi i)^{n-1}} \times \\
& \int \limits_{\mathbb{T}} \sum \limits_{\sigma \in S_n} \left(\beta_{\sigma(1)}^{\lambda_1+\cdots+\lambda_{n-1}+\log|t_1|_p+\cdots+\log|t_{n-1}|_p} \cdots \beta_{\sigma(n-1)}^{\lambda_1+\log|t_1|_p} \right) \frac{d \beta_1 \cdots d \beta_{n-1}}{\beta_1 \cdots \beta_{n-1}} \\
& = \frac{ \delta^{1/2}(t)}{\delta^{1/2}(p^{\lambda})}  f_{\lambda_1,\cdots,\lambda_{n-1}}(t) \\
& = f_{\lambda_1,\cdots,\lambda_{n-1}}(t)
\end{align*}
where the last equality follows from the definition of $f_{\lambda_1,\cdots,\lambda_{n-1}}(t)$.
\end{proof}

\section{Integral representations of local $L$-factors}

In this final section, we obtain integral representations of the local $L$-factors of symmetric $d$th power $L$-functions of $\text{GL}(2)$ representations, with $d \leq 4$.
Let $L(s, \text{Sym}^d \, \pi)$ be the symmetric $d$th power $L$-function associated to an irreducible automorphic representation $\pi$ of $\text{GL}(2,\mathbb{A}_\qq)$. This $L$-function has local factors (at the unramified places $p$) given by
\begin{equation*}
L_p(s,\text{Sym}^d \, \pi) = \prod_{i=0}^{d} (1-\alpha^{d-2i}p^{-s})^{-1} =: h_{s,p,d}(\alpha).
\end{equation*}

We shall obtain an integral representation for the $L$-factors $L_p(s,\pi)$ using Theorem~\ref{transformthm}. 

Assume $\mathfrak{R}(s) > 1$. We start by computing $(h_{s,p,d})^{\flat}$ explicitly. By definition,

\begin{align*}
(h_{s,p,d})^{\flat}(t) & = \frac{1}{4\pi i} \int \limits_{\lvert \beta \rvert_{\cc} = 1} h_{s,p,d}(\beta) W_{1/\beta}(t) (\beta - \beta^{-1})(\beta^{-1} - \beta) \frac{d \beta}{\beta} \\
& = \frac{\lvert t_1 \rvert_p^{1/2}}{4\pi i} \int \limits_{\lvert \beta \rvert_{\cc} = 1} h_{s,p,d}(\beta) \left( \beta^{\lambda+1} - \beta^{-(\lambda+1)} \right) \left(\beta^{-1} - \beta \right) \frac{d \beta}{\beta}
\end{align*}
with $t = \begin{pmatrix}
t_1 & 0  \\
0 & 1 \end{pmatrix}$ and $\lvert t_1 \rvert_p = p^{-\lambda}$ for $\lambda \geq 0$.

We now want to use the residue theorem to evaluate the integral above. To simplify the calculation of the residues of the integrand we first note that

\begin{align*}
\int \limits_{\lvert \beta \rvert_{\cc} = 1} h_{s,p,d}(\beta) \beta^{-(\lambda+1)} \left(\beta^{-1} - \beta \right) \frac{d \beta}{\beta}  = - \int \limits_{\lvert \beta \rvert_{\cc} = 1} h_{s,p,d}(\beta) \left( \beta^{\lambda+1} \right) \left(\beta^{-1} - \beta \right) \frac{d \beta}{\beta}
\end{align*}

via the change of variable $\beta \mapsto \frac{1}{\beta}$. Therefore,

\begin{equation}
\label{eq:integral}
(h_{s,p,d})^{\flat}(t) = \frac{\lvert t_1 \rvert_p^{1/2}}{2\pi i} \int \limits_{\lvert \beta \rvert_{\cc} = 1} h_{s,p,d}(\beta)  \left(\beta^{\lambda-1} - \beta^{\lambda+1} \right) d \beta.
\end{equation}

Let $i_{s,p,d}(\beta)$ be the integrand of the integral in Equation~\ref{eq:integral}. By the definition of $h_{s,p,d}$, the function $i_{s,p,d}(\beta)$ has possible poles inside the unit circle at

\begin{equation*}
\beta = e^{\frac{2\pi i l}{k}} p^{-\frac{s}{k}}
\end{equation*}

for $0 < k \leq d$, $k \equiv d \pmod 2$ and $0 \leq l < k$. We will compute the residues at these poles for $1 \leq d \leq 4$, in order to compute the integral in Equation~\ref{eq:integral}.

\begin{itemize}
\item $d=1$:

In this case, the function $i_{s,p,1}(\beta)$ only has a pole at $\beta = p^{-s}$ and

\begin{equation*}
\text{Res}_{\beta = p^{-s}} \, i_{s,p,1}(\beta) = p^{-s \lambda},
\end{equation*}
which implies that 
\begin{equation}
\label{firstlift}
(h_{s,p,1})^{\flat}(t) = \lvert t_1 \rvert_p^{s+1/2}.
\end{equation}

\item $d=2$:

In the case of a symmetric square $L$-function, the integrand $i_{s,p,2}(\beta)$ has poles at $\beta = \pm p^{-s/2}$. Their residues are given by
\begin{align*}
&\text{Res}_{\beta = p^{-s/2}} \, i_{s,p,2}(\beta)  = \frac{p^{-s\lambda/2}}{2(1-p^{-2s})},\\
&\text{Res}_{\beta = -p^{-s/2}} \, i_{s,p,2}(\beta) = -\frac{p^{-s\lambda/2}}{2(1-p^{-2s})}.
\end{align*}

Adding the two residues we infer that

\begin{equation}
\label{squarelift}
(h_{s,p,2})^{\flat}(t)  = \begin{cases} \frac{\lvert t_1 \rvert_p^{(s+1)/2}}{1-p^{-2s}}, & \mbox{if } \lambda \mbox{ even}, \\ 0, & \mbox{if } \lambda \mbox{ odd}.  \end{cases}
\end{equation}

\item $d=3$:

For the symmetric cube $L$-function, the poles of the integrand $i_{s,p,3}(\beta)$ are at $\beta = p^{-s}$ and at $\beta = e^{\frac{2\pi i l}{3}} p^{-s/3}$ for $l = 0,1,2$.
The residue at $\beta = p^{-s}$ is given by
\begin{equation*}
\text{Res}_{\beta = p^{-s}} \, i_{s,p,3}(\beta)  = \frac{-p^{-s(\lambda+2)}}{(1-p^{-2s})(1-p^{-4s})}.
\end{equation*}
For clarity's sake we won't write down the values of the residues at the remaining three poles, but simply note that their sum is equal to
\begin{equation*}
\frac{1}{(1-p^{-2s})(1-p^{-4s})}
\begin{cases} 
p^{-s\lambda/3}, & \mbox{if } \lambda \equiv 0 \bmod 3, \\ 
p^{-s(\lambda+8)/3}, & \mbox{if } \lambda \equiv 1 \bmod 3,\\
p^{-s(\lambda+4)/3}, & \mbox{if } \lambda \equiv 2 \bmod 3. 
\end{cases}
\end{equation*}
We remark that calculating these residues is a straightforward, albeit tedious, exercise.
Finally, we conclude that $(h_{s,p,3})^{\flat}(t)$ is equal to
\begin{equation}
\label{cubelift}
\frac{1}{(1-p^{-2s})(1-p^{-4s})}\begin{cases} \lvert t_1 \rvert_p^{s/3+1/2}-\lvert t_1 \rvert_p^{s-3/2}, & \mbox{if } \lambda \equiv 0 \bmod 3, \\
\lvert t_1 \rvert_p^{s/3-13/6}-\lvert t_1 \rvert_p^{s-3/2}, & \mbox{if } \lambda \equiv 1 \bmod 3,\\
\lvert t_1 \rvert_p^{s/3-5/6}-\lvert t_1 \rvert_p^{s-3/2}, & \mbox{if } \lambda \equiv 2 \bmod 3.  
\end{cases}
\end{equation}

\item $d=4$:

In this case, the integrand $i_{s,p,4}(\beta)$ has poles at $\beta = \pm p^{-s/2}$ and at $\beta = \pm p^{-s/4}, \pm i p^{-s/4}$. The sum of the residues at the first two poles is
\begin{equation*}
\frac{1}{(1-p^{-2s})(1-p^{-3s})}
\begin{cases} 
-p^{-s(\lambda+2)/2}, & \mbox{if } \lambda \equiv 0 \bmod 2, \\ 
0, & \mbox{if } \lambda \equiv 1 \bmod 2. 
\end{cases}
\end{equation*}
and the sum of the residues at the last four poles is
\begin{equation*}
\frac{1}{(1-p^{-2s})(1-p^{-3s})}
\begin{cases} 
p^{-s\lambda/4}, & \mbox{if } \lambda \equiv 0 \bmod 4, \\ 
p^{-s(\lambda+6)/4}, & \mbox{if } \lambda \equiv 2 \bmod 4,\\
0, & \mbox{otherwise.} 
\end{cases}
\end{equation*}

Adding up all the residues one obtains that $(h_{s,p,4})^{\flat}(t)$ is equal to
\begin{equation}
\label{fourthlift}
\frac{1}{(1-p^{-2s})(1-p^{-3s})}\begin{cases} \lvert t_1 \rvert_p^{s/4+1/2}-\lvert t_1 \rvert_p^{s/2-1/2}, & \mbox{if } \lambda \equiv 0 \bmod 4, \\
\lvert t_1 \rvert_p^{s/4-1}-\lvert t_1 \rvert_p^{s/2-1/2}, & \mbox{if } \lambda \equiv 2 \bmod 4,\\
0, & \mbox{otherwise.}  
\end{cases}
\end{equation}
\end{itemize}
%
%
%
%
%
For large enough $\mathfrak{R}(s)$, since the functions $(h_{s,p,d})^{\flat}(t)$ satisfy the convergence condition in Remark~\ref{convergence}, it follows from Theorem~\ref{transformthm} that

\begin{equation*}
L_p(s,\text{Sym}^d \, \pi) = \int \limits_{\mathbb{Q}_p^{\times}} (h_{s,p,d})^{\flat}(t) W_{\alpha}(t)  d^{\times}t,
\end{equation*}
for $1 \leq d \leq 4$, where the functions $(h_{s,p,d})^{\flat}(t)$ are explicitly given by Equations~\ref{firstlift}-\ref{fourthlift}.

This approach should provide (conditionally on convergence issues) integral representations for $L_p(s, \text{Sym}^d \, \pi)$ for $d > 4$ as well. However, the residue calculations appear to become rather complicated. 

\section*{Acknowledgments}

The author would like to thank Dorian Goldfeld for helpful discussions, as well as for the comments on earlier drafts of this work. The author was partially supported by the FCT doctoral grant SFRH/BD/68772/2010.

\bibliography{biblio}
\bibliographystyle{plain}

\end{document}